\documentclass[]{amsart}   
\usepackage{amsmath}
\usepackage[mathscr]{eucal} 
\usepackage{amssymb}
\usepackage{latexsym}
\usepackage{amsthm} 
\theoremstyle{plain}
\newtheorem{theorem}{Theorem}[section]
\newtheorem{lemma}[theorem]{Lemma}  

\newtheorem{theorema}{Theorem A\!\!}
\newtheorem{theoremb}{Theorem B\!\!}

\newtheorem{proposition}[theorem]{Proposition}

\newtheorem{corollary}[theorem]{Corollary}

\newcommand{\supp}{\mathop{\mathrm{supp}}\nolimits} 
\newcommand{\sgn}{\mathop{\mathrm{sgn}}\nolimits}

\theoremstyle{definition}

\numberwithin{equation}{section} 
\newtheorem{remark}[theorem]{Remark} 
\newtheorem{aexample}{Example\!\!}

\theoremstyle{remark}

\def\Xint#1{\mathchoice
{\XXint\displaystyle\textstyle{#1}}%
{\XXint\textstyle\scriptstyle{#1}}%
{\XXint\scriptstyle\scriptscriptstyle{#1}}%
{\XXint\scriptscriptstyle\scriptscriptstyle{#1}}%
\!\int}
\def\XXint#1#2#3{{\setbox0=\hbox{$#1{#2#3}{\int}$}
\vcenter{\hbox{$#2#3$}}\kern-.5\wd0}}

\def\dashint{\Xint-}

\title[ Sobolev spaces with non-isotropic dilations
  ]
{Sobolev spaces with non-isotropic dilations and 
 square functions of Marcinkiewicz type}  
\author{Shuichi Sato} 
 
\begin{document} 

\address{Department of Mathematics,
Faculty of Education, Kanazawa University, Kanazawa 920-1192, Japan}
\email{shuichipm@gmail.com}
\begin{abstract} 
We consider the weighted 
Sobolev spaces associated with non-isotropic dilations 
of Calder\'on-Torchinsky and characterize the spaces 
 by the square functions of Marcinkiewicz type including those defined 
with repeated uses of averaging operation. 
\end{abstract}
  \thanks{ 
  \endgraf
    }
\thanks{The author is partly supported
by Grant-in-Aid for Scientific Research (C) No. 20K03651, Japan Society for the  Promotion of Science.} 
\maketitle 

\section{Introduction} \label{s1} 

Let $B(x,t)$ be a ball in $\Bbb R^n$ with radius $t$ centered at
 $x$ and for $0<\alpha<2$ let 
\begin{equation}\label{1.2}  
V_\alpha(f)(x)=\left(\int_0^\infty\left|f(x)-\dashint_{B(x,t)}f(y)\, dy
\right|^2\frac{dt}{t^{1+2\alpha}} \right)^{1/2},  
\end{equation}  
where $\dashint_{B(x,t)} f(y)\,dy$ denotes    
$|B(x,t)|^{-1}\int_{B(x,t)} f(y)\,dy$ and $|B(x,t)|$  
the Lebesgue measure. 
In \cite{AMV} the operator $V_1$ was used to characterize the Sobolev
 space $W^{1,p}(\Bbb R^n)$ as follows.    
\begin{theorema}\label{thma} Let $1<p<\infty$. Then, 
$f$ belongs to $W^{1,p}(\Bbb R^n)$ if and only if 
 $f\in L^p(\Bbb R^n)$ and $V_1(f)\in L^p(\Bbb R^n);$  
furthermore, 
$$\|V_1(f)\|_p\simeq \|\nabla f\|_p,   $$  
which  means that 
 there exist positive constants $c_1$, $c_2$ independent of $f$ 
 such that
$$c_1  \|V_1(f)\|_p\leq \|\nabla f\|_p \leq c_2\|V_1(f)\|_p.$$ 
\end{theorema} 
\par 
Let $\mathscr S(\Bbb R^n)$ be the Schwartz class  of 
rapidly decreasing smooth functions on $\Bbb R^n$.  
Define 
$$\mathscr S_0(\Bbb R^n)=\{f\in \mathscr S(\Bbb R^n): 
\text{$\hat{f}$ vanishes near the origin} \},   $$ 
where the Fourier transform $\hat{f}$ is defined as    
$$\hat{f}(\xi)=\int_{\Bbb R^n} f(x)e^{-2\pi i\langle x,\xi\rangle}\, dx, 
\quad  \langle x,\xi\rangle=\sum_{k=1}^n x_k\xi_k.  
$$  
We also write $\mathscr F(f)$ for $\hat{f}$. 
For $0<\alpha<n$, $n\geq 2$,  let  $I_\alpha$ be the Riesz potential operator 
defined by 
\begin{equation}\label{1.04}
\mathscr F(I_\alpha(f))(\xi)=(2\pi|\xi|)^{-\alpha} \hat{f}(\xi),   
\quad f\in \mathscr S_0  
\end{equation}  
(see \cite[Chap. V]{St}).   
Let  
\begin{equation}\label{1.3} 
S_\alpha(f)(x)=\left(\int_0^\infty \left|I_\alpha(f)(x)-
\dashint_{B(x,t)}I_\alpha(f)(y)\, dy\right|^2 
\frac{dt}{t^{1+2\alpha}}\right)^{1/2},  
\end{equation}  
Then we also find the following result in \cite{AMV}. 
\begin{theoremb}\label{thmb} 
  Let $0<\alpha<2$ and $1<p<\infty$. Then 
$$\|S_\alpha (f)\|_p \simeq \|f\|_p.   $$
\end{theoremb}  
Theorem A can be derived from this result with $\alpha=1$ when $n\geq 2$. 
\par  
The operator $S_\alpha$ is a kind of the Littlewood-Paley operators.  
Let $\psi\in L^1(\Bbb R^n)$  satisfy   
\begin{equation}\label{cancell}
\int_{\Bbb R^n} \psi (x)\,dx = 0.  
\end{equation} 
Put $\psi_t(x)=t^{-n}\psi(t^{-1}x)$.  
Then  
the Littlewood-Paley function on $\Bbb R^n$ is defined by 
\begin{equation}\label{lpopeu}
g_{\psi}(f)(x) = \left( \int_0^{\infty}|f*\psi_t(x)|^2
\,\frac{dt}{t} \right)^{1/2}.    
\end{equation} 
We can see that $S_\alpha(f)=g_{\psi^{(\alpha)}}(f)$, where 
\begin{equation} \label{eee1.6} 
\psi^{(\alpha)}(x)=L_\alpha(x)-\Phi*L_\alpha(x),  
\end{equation}   
with 
$$L_\alpha(x)=\tau(\alpha)|x|^{\alpha-n}, \quad  
\tau(\alpha)=\frac{\Gamma\left(n/2-\alpha/2\right)}
{\pi^{n/2}2^\alpha \Gamma\left(\alpha/2\right)} $$ 
and $\Phi=\chi_0$, $\chi_0=|B(0,1)|^{-1}\chi_{B(0,1)}$ 
($\chi_E$ denotes the characteristic function of a set $E$).    
We note that $\mathscr F(L_\alpha)(\xi)=(2\pi|\xi|)^{-\alpha}$, $0<\alpha<n$. 
\par 
The square function $S_1(f)$ is closely related to the function of 
Marcinkiewicz on $\Bbb R^1$, which is defined by 
$$\mu(f)(x) = \left( \int_0^{\infty}|F(x+t) + F(x-t) - 2F(x)|^2\,\frac{dt}{t^3}
 \right)^{1/2}, $$ 
 where  $F(x) = \int_{-\infty}^xf(y)\,dy$, $f\in \mathscr S(\Bbb R)$.  
  It is known that   
\begin{equation}\label{1.1}
\|\mu(f)\|_p\simeq \|f\|_p, 
\end{equation} 
for $1<p<\infty$. Also, we consider a variant of $\mu(f)$, which can be 
regarded as an analogue of $S_1$ in the one-dimensional case: 
$$\nu(f)(x) = \left( \int_0^{\infty}|F(x)-F*\Phi_t(x)|^2\,\frac{dt}{t^3}
 \right)^{1/2},   $$ 
where $\Phi=(1/2)\chi_{[-1,1]}$.  
It is known that $\mu(f)=g_\psi(f)$ with $\psi(x)=\chi_{[-1,1]}(x)\sgn(x)$. 
By inspection, we see that $\nu(f)=g_{\psi^{(0)}}(f)$, where 
$\psi^{(0)}(x)=(1/2)\psi(x)-(1/2)\psi^{(1)}(x)$ with 
$\psi^{(1)}(x)=x\chi_{[-1,1]}(x)$. 
This would indicate that 
the square functions $\mu(f)$ and $\nu(f)$ are intimately related.      
For the function of Marcinkiewicz  we refer to \cite{M},    
Zygmund \cite{Z},  Waterman \cite{Wa}.       
\par 
An interesting feature of Theorem A is that it suggests  
the possibility of defining the Sobolev space analogous to $W^{1,p}(\Bbb R^n)$ 
in metric measure spaces in a reasonable way. 
In this note, we shall extend Theorem A to the case of the weighted 
Sobolev spaces with the parabolic metrics of Calder\'{o}n-Torchinsky 
\cite{CT, CT2}.  
\par 
Let $P$ be an $n\times n$ real matrix, $n\geq 2$,  such that 
\begin{equation}\label{pe1}
\langle Px,x\rangle \geq \langle x,x\rangle \quad \text{for all $x\in 
\Bbb R^n$. }
\end{equation} 
A dilation group 
$\{\delta_t\}_{t>0}$ on $\Bbb R^n$ is defined by $\delta_t=
t^P=\exp((\log t) P)$. 
\par 
It is known that $|\delta_tx|=\langle \delta_tx,\delta_tx\rangle^{1/2}$
 is strictly increasing as a function of $t$ on $(0,\infty)$ 
when $x\neq 0$.  
Let $\rho(x)$, $x\neq 0$, to be the unique positive real 
number $t$ such that $|\delta_{t^{-1}}x|=1$ and let $\rho(0)=0$.  
Then the norm function $\rho$ is continuous on $\Bbb R^n$ and 
infinitely differentiable in $\Bbb R^n\setminus \{0\}$ and satisfies that 
$\rho(A_tx)=t\rho(x)$, $t>0$, $x\in \Bbb R^n$.   
 We have the following properties of $\rho(x)$ 
(see \cite{CT, Ca}):  
\begin{enumerate} 
\item[(1)]  $\rho(-x)=\rho(x)$ for all $x\in \Bbb R^n$; 
\item[(2)]  $\rho(x+y)\leq \rho(x)+\rho(y)$ for all $x, y \in \Bbb R^n$; 
\item[(3)]  $\rho(x)\leq 1$ if and only if $|x|\leq 1$; 
\item[(4)]   $c_1\rho(x)^{\tau_1}\leq |x|\leq \rho(x)$ when $|x|\leq 1$ for 
some $c_1, \tau_1>0$; 
\item[(5)]   $\rho(x)\leq |x|\leq c_2\rho(x)^{\tau_2}$ when $|x|\geq 1$ for 
some $c_2, \tau_2>0$.  
\end{enumerate} 
Also, 
\begin{enumerate} 
\item[(a)]   $|\delta_tx|\geq t|x|$ for all $x\in \Bbb R^n$ and $t\geq 1$; 
\item[(b)]   $|\delta_tx|\leq t|x|$ for all $x\in \Bbb R^n$ and $0<t\leq 1$.  
\end{enumerate} 
Let $\delta_t^*$ denote the adjoint of $\delta_t$. Then,  
  we can also consider similarly a norm function $\rho^*(x)$ associated with 
 the dilation group $ \{\delta_t^*\}_{t>0}$, and we have properties of 
$\rho^*(x)$ and $\delta_t^*$ analogous to those 
 of $\rho(x)$ and $\delta_t$ above.  
It is known that 
a polar coordinates expression for the Lebesgue measure  
\begin{equation}\label{1.9polar}
\int_{\Bbb R^n}f(x)\,dx=\int_0^\infty\int_{S^{n-1}} f(\delta_t\theta)
t^{\gamma-1}s(\theta)\, d\sigma(\theta)\,dt  
\end{equation}  
holds,  
where $\gamma=\text{{\rm trace} $P$}$ and 
$s$ is a strictly positive $C^\infty$ function on 
$S^{n-1}=\{|x|=1\}$ and 
$d\sigma$ is the Lebesgue surface measure on $S^{n-1}$ 
(see \cite{FR, Ri, SW}).  
We note that the condition \eqref{pe1} implies that all 
eigenvalues of $P$ have real parts greater than or equal to $1$ 
(see \cite[pp. 3--4]{CT}, \cite[p. 137]{LA}). 
So we have $\gamma\geq n$.    
\par     
Let 
\begin{equation}\label{1.10ball}
B(x,t)=\{y\in\Bbb R^n: \rho(x-y)<t\} 
\end{equation}  
be a  ball  with respect to $\rho$ (a $\rho$-ball) in $\Bbb R^n$ 
with radius $t$ centered at $x$. 
We say that a weight function $w$ belongs to the  Muckenhoupt class $A_p$, 
 $1<p< \infty$,   if 
 $$[w]_{A_p}=  
 \sup_B \left(|B|^{-1} \int_B w(x)\,dx\right)\left(|B|^{-1} \int_B
w(x)^{-1/(p-1)}dx\right)^{p-1} < \infty, $$
where the supremum is taken over all $\rho$-balls $B$ in $\Bbb R^n$. 
The Hardy-Littlewood maximal operator $M$ is defined as 
$$M(f)(x)=\sup_{x\in B}|B|^{-1}\int_B|f(y)|\,dy,   $$ 
where the supremum is taken over all $\rho$-balls $B$ in $\Bbb R^n$ 
containing $x$.  
The class $A_1$ is defined to be the family of weight functions $w$ 
such that $M(w)\leq Cw$ almost everywhere; the infimum of all such $C$  
will be denoted by $[w]_{A_1}$.  
We denote by $L^p_w$ (we also write $L^p(w)$ for $L^p_w$)   
the weighted $L^p$ space with the norm defined as 
 $$  \|f\|_{L^p_w}=\|f\|_{L^p(w)}= \left(\int_{\Bbb R^n} |f(x)|^pw(x)
 \,dx\right)^{1/p}. $$  
 See \cite{C}, \cite{CF}, \cite{GR}, \cite{ST}   
 for results related to the weight class $A_p$.   
The following results are known and useful. 
\begin{proposition}\label{prop0} 
Let $1<p<\infty$, $w\in A_p$.  
\begin{enumerate} 
\item[(i)] The space $\mathscr S_0$ is dense in $L^p_w$.   
\item[(ii)] The maximal operator $M$ is bounded on $L^p_w$. 
\item[(iii)] If $\varphi\in \mathscr S$, then $\sup_{t>0}|f*\varphi_t|\leq 
CM(f)$. Here and in what follows 
$\varphi_t(x)=t^{-\gamma}\varphi(\delta_t^{-1}x)$. 
\item[(iv)]  $\mathscr F(g*\varphi_t)(\xi)=
\hat{g}(\xi)\hat{\varphi}(\delta_t^*\xi)$ 
for $g, \varphi\in \mathscr S$.   
\end{enumerate}   

\end{proposition}
\par 
Let $\beta \in \Bbb R$ and define the Riesz potential operator 
$\mathcal I_\beta$ associated with the dilations $\delta_t^*$ by 
\begin{equation} 
\mathscr{F}(\mathcal I_\beta(f))(\xi)=\rho^*(\xi)^{-\beta}\hat{f}(\xi) 
\end{equation} 
for $f\in \mathscr S_0$.   
Let $1<p<\infty$, $\alpha>0$ and $w\in A_p$. 
Define the weighted parabolic Sobolev space 
$W_w^{\alpha, p}$ by 
\begin{equation}\label{def1}
W_w^{\alpha, p}=\{f\in L^p_w : \text{$f=\mathcal{I}_\alpha(g)$ 
for some $g\in L^p_w$}\}, 
\end{equation} 
where $f=\mathcal{I}_\alpha(g)$ means that 
\begin{equation*} 
\int_{\Bbb R^n} f(x)h(x)\, dx= \int_{\Bbb R^n} g(x)\mathcal{I}_\alpha(h)\, dx
\quad \text{for all $h\in \mathscr S_0$}.  
\end{equation*} 
We note that the function $g \in L^p_w$ is uniquely determined by $f$, since 
$\mathcal{I}_\alpha$ is a bijection on $\mathscr S_0$ and $\mathscr S_0$ is 
dense in $L^{p'}(w^{-p'/p})$, the dual space of $L^p(w)$,  where $1/p+1/p'=1$. 
We write $g=\mathcal{I}_{-\alpha}(f)$. 
For $f\in W_w^{\alpha, p}$ we define 
\begin{equation}\label{def2} 
\|f\|_{p,\alpha,w}= \|f\|_{p,w}+ \|\mathcal{I}_{-\alpha}(f)\|_{p,w}. 
\end{equation}  
\par 
We have analogues of Theorems A and B in the case of non-isotropic dilations 
$\delta_t$ with weights. 
 Let $B(x,t)$ be as in \eqref{1.10ball} and 
 \begin{equation}\label{cevv2} 
 B_\alpha (f)(x)=\left(\int_0^\infty \left|f(x)-
 \dashint_{B(x,t)}f(y)\, dy\right|^2 
\frac{dt}{t^{1+2\alpha}}\right)^{1/2},  \quad \alpha>0.  
 \end{equation} 
\begin{theorem} \label{thm1} 
 Suppose that 
 $1<p<\infty$, $w\in A_p$ and $0<\alpha<2$.  
  Then 
$f\in W^{\alpha, p}_w$ if and only if $f\in L^p_w$ and 
$B_\alpha(f) \in L^p_w;$ also,  
$$\|\mathcal{I}_{-\alpha}(f)\|_{p,w} \simeq \|B_\alpha(f)\|_{p,w}.  $$
\end{theorem} 
Let  
\begin{equation}\label{cb2} 
C_\alpha(f)(x)=\left(\int_0^\infty \left|\mathcal{I}_\alpha(f)(x)-
\dashint_{B(x,t)}\mathcal{I}_\alpha(f)(y)\, dy\right|^2 
\frac{dt}{t^{1+2\alpha}}\right)^{1/2}.   
\end{equation}  
Then, Theorem \ref{thm1} can be derived from the following result. 
\begin{theorem}\label{thm2} Let $1<p<\infty$, $w\in A_p$, $0<\alpha<2$ 
and  let $C_\alpha$ be as in \eqref{cb2}.  Then 
$$\|C_\alpha (f)\|_{p,w} \simeq \|f\|_{p,w}, 
\quad f\in \mathscr S_0(\Bbb R^n).   $$
\end{theorem}  
The range of $\alpha$ in Theorem \ref{thm1} will be extended in 
Theorem \ref{thm4.2} in Section \ref{s4} by considering square functions 
with repeated uses of averaging operation $\dashint_Bf$. 
\par 
We consider square functions generalizing $B_\alpha$ and $C_\alpha$ in 
\eqref{cevv2} and \eqref{cb2}.  
Let $\Phi$ be a bounded function on $\Bbb R^n$ with compact support. 
We say 
$\Phi \in \mathcal M^\alpha$, $\alpha\geq 0$,  
if $\Phi$ satisfies 
\begin{enumerate}   
\item[(i)] 
$\int_{\Bbb R^n} \Phi(x)\, dx=1$; 
\item[(ii)] 
if $\alpha\geq 1$, 
\begin{equation}\label{moment} 
\int_{\Bbb R^n} \Phi(x)x^a \, dx=0  \quad  \text{for all 
multi-indices $a$ with $1\leq |a|\leq [\alpha]$,}   
\end{equation} 
where 
$x^a= x_1^{a_1}\dots x_n^{a_n}$ 
with $a=(a_1, \dots ,a_n)$,  $|a|=a_1+\dots +a_n$,  $a_j\in \Bbb Z$ 
(the set of integers), $a_j\geq 0$, $1\leq j\leq n$, and 
$[\alpha]=\max\{k\in \Bbb Z: k\leq \alpha\}$.        
\end{enumerate} 
\par 
We note that $\mathcal M^\alpha \subset \mathcal M^\beta$ if $\alpha\geq \beta$ and 
$\mathcal M^\alpha =\mathcal M^j$  if  $j\leq \alpha<j+1$, 
$j\geq 0$, $j\in \Bbb Z$.  
If $\Phi$ is even and $1\leq \alpha<2$, we have \eqref{moment}.  
In particular,  
$\chi_0=|B(0,1)|^{-1}\chi_{B(0,1)}\in \mathcal M^\alpha$ 
for  $0\leq \alpha<2$.  
\par 
 Let $\Phi\in \mathcal M^\alpha$ and 
 \begin{equation}\label{cev2} 
 G_\alpha (f)(x)=\left(\int_0^\infty \left|f(x)-
 \Phi_t*f(x)\right|^2 \frac{dt}{t^{1+2\alpha}}\right)^{1/2},  \quad \alpha>0. 
 \end{equation} 
We note that if $\Phi=\chi_0$ in \eqref{cev2}, we have  
$B_\alpha$ of \eqref{cevv2}. 
 Also,  let $\Phi\in \mathcal M^\alpha$ and  
 \begin{equation}\label{cev} 
 H_\alpha (f)(x)=\left(\int_0^\infty \left|\mathcal{I}_\alpha(f)(x)-
\Phi_t*\mathcal{I}_\alpha(f)(x)\right|^2 
\frac{dt}{t^{1+2\alpha}}\right)^{1/2},  \quad 0<\alpha<\gamma.  
 \end{equation} 
If we set $\Phi=\chi_0$ in \eqref{cev},  we get  $C_\alpha$ of \eqref{cb2} 
for $0<\alpha<2$.   
 \par 
We prove the following.  
\begin{theorem}\label{thm3} 
 Let $H_\alpha$ be as in \eqref{cev} and $0<\alpha<\gamma$, 
$1<p<\infty$, $w\in A_p$.   
  Then 
$$\|H_\alpha (f)\|_{p,w} \simeq \|f\|_{p,w}, \quad f\in 
\mathscr S_0(\Bbb R^n).     $$  
\end{theorem}  
\par 
Applying Theorem \ref{thm3}, we have the following. 
\begin{theorem}\label{thm4} 
 Suppose that $1<p<\infty$, $w\in A_p$ and $0<\alpha<\gamma$.  
Let $G_\alpha$ be as in \eqref{cev2}.  Then 
$f\in W^{\alpha, p}_w$ if and only if $f\in L^p_w$ and 
$G_\alpha(f) \in L^p_w;$  furthermore,  
$$\|\mathcal{I}_{-\alpha}(f)\|_{p,w}
\simeq \|G_\alpha(f)\|_{p,w}.  $$
\end{theorem}  
Theorems \ref{thm1} and \ref{thm2} follow from Theorems \ref{thm4} 
and \ref{thm3}, respectively. The proofs of Theorems \ref{thm3} 
and \ref{thm4} will be given in Section \ref{s3}. 
To prove Theorem \ref{thm3},  
we consider Littlewood-Paley functions
\begin{equation}\label{lpop}
g_{\psi}(f)(x) = \left( \int_0^{\infty}|f*\psi_t(x)|^2
\,\frac{dt}{t} \right)^{1/2},   
\end{equation} 
where $\psi_t(x)=t^{-\gamma}\psi(\delta_t^{-1}x)$ with   
 $\psi\in L^1(\Bbb R^n)$ satisfying \eqref{cancell}. 
Then we can see that $H_\alpha(f)=g_{\psi^{(\alpha)}}$ for 
some $\psi^{(\alpha)}$ analogous to the one in \eqref{eee1.6}.  
We shall prove Theorem \ref{thm3} by 
applying Theorem \ref{thm5} below in Section \ref{s2}, which is a 
result for parabolic Littlewood-Paley functions  complementing 
the boundedness result given in \cite{Sacze} and generalizing 
 \cite[Corollary $2.11$]{Sa4} to the case of non-isotropic dilations. 
\par 
The proof of Theorem \ref{thm5} will be completed by applying 
Theorem \ref{thm7} below in Section \ref{s2}, which provides the 
estimates 
\begin{equation}\label{eee1.20} 
\|f\|_{p,w}\leq C\|g_\psi(f)\|_{p,w} 
\end{equation} 
under certain conditions. 
Theorem \ref{thm7} is proved by Corollary \ref{inverse} 
 in Section \ref{s2}, which is a result on the invertibility of 
Fourier multipliers homogeneous of degree $0$ with respect to $\delta_t^*$ 
generalizing \cite[Corollary $2.6$]{Sa4} to the case of general homogeneity. 
Corollary \ref{inverse} will follow from a more general result (Theorem 
\ref{thm6}).  
\par 
Here we review some recent developments of the theory related to 
the results given in this note 
after the article \cite{AMV} (see also the remarks at the end of this note). 
Theorem A was generalized to the weighted Sobolev spaces by \cite{HL}. Also, 
Theorems A and B were extended to the weighted Sobolev spaces in \cite{Sa2} by 
applying a theorem in \cite{Sa} for the boundedness of 
Littlewood-Paley functions $g_\psi$ in \eqref{lpopeu}  
on the weighted $L^p$ spaces, which is 
partly a special case of Theorem \ref{thm5} in Section \ref{s2}.  
In \cite{Sa2} it was shown that the theorem of \cite{Sa} is 
particularly suitable for handling the square functions 
in Theorem \ref{thm3} above for the case of the Euclidean structures 
(with the Euclidean norm and the ordinary dilation). 
Some results of \cite{Sa2} were generalized in \cite{Sa4} 
by introducing the function class $\mathcal M^\alpha$ and 
by proving the weighted $L^p$ norm equivalence between $g_\psi(f)$ in 
\eqref{lpopeu} and $f$,  
part of which was not included in \cite{Sa}; 
the estimates in \eqref{eee1.20} in the case of the Euclidean structures  
for sufficiently large class of $\psi$ and $p\in (1, \infty)$, 
$w\in A_p$ were absent from   \cite{Sa}. 
In \cite{Sa5} and \cite{Sa4}, discrete parameter versions of 
Littlewood-Paley functions $g_\psi(f)$ in \eqref{lpopeu} of the form   
\begin{equation*} 
\Delta_\psi(f)(x)
=\left(\sum_{k=-\infty}^{\infty}\left|f*\psi_{2^k}(x)\right|^2  \right)^{1/2}  
\end{equation*}
 are also considered to characterize 
the Sobolev spaces. 
See also \cite{HL} and \cite{Sa6} for applications of the 
square function   
\begin{equation*}\label{esr1.1} 
D^\alpha(f)(x)=\left(\int_0^\infty \left|t^{-\alpha}\int_{S^{n-1}}
(f(x-t\theta)-f(x))\, d\sigma(\theta) \right|^2\, \frac{dt}{t}\right)^{1/2}    
\end{equation*} 
in the theory of Sobolev spaces. 
\par 
In Section \ref{s4}, we shall establish another characterization of 
the Sobolev spaces $W^{\alpha,p}_w$ similar to Theorem \ref{thm1} 
(Theorem \ref{thm4.2}), which is novel even in the case of the Euclidean 
structures. In Theorem \ref{thm1}, the averaging operator 
$\dashint_{B}f$ is used to define the square function $B_\alpha(f)$ 
in \eqref{cevv2} which 
is applied to characterize $W^{\alpha,p}_w$ for $\alpha\in (0, 2)$. 
In Theorem \ref{thm4.2} we shall extend the range of $\alpha$ 
 by introducing square functions which are defined with repeated uses of 
averaging operation.   
\par 
Finally, in Section \ref{s5} we shall illustrate by example how 
 the Sobolev spaces $W_w^{\alpha, p}$ defined above can be characterized 
by distributional derivatives in some cases,  
by the arguments similar to the one in the proof of 
\cite[Theorem $3$ of Chap. V]{St}.

\section{Invertibility of  Fourier multipliers homogeneous  
with respect to $\delta_t^*$ and Littlewood-Paley operators}\label{s2}

We consider a majorant of $\psi$ defined by 
$$H_{\psi}(x) =h(\rho(x))= \sup_{\rho(y)\geq \rho(x)}|\psi(y)|$$
and  two seminorms $B_{\epsilon}$ and $D_u$ defined as 
$$B_{\epsilon}(\psi) = \int_{|x|>1}|\psi(x)|\,|x|^{\epsilon}\,dx   
\qquad \text{ for \quad  $\epsilon >0$} , $$
$$D_u(\psi)=\left(\int_{|x|<1}|\psi(x)|^u\,dx\right)^{1/u} \qquad \text{for \quad $u>1$}.$$  
 In proving Theorem \ref{thm3} we apply the following result.  
\begin{theorem}\label{thm5}  
Suppose that  $\psi \in L^1(\Bbb R^n)$ satisfies \eqref{cancell}. 
Let $\epsilon>0, u>1$ and $C_j>0$, $1\leq j\leq 3$. 
Suppose that 
\begin{enumerate}
\item[(1)]  $B_{\epsilon}(\psi) \leq C_1;$  
\item[(2)]  $D_u(\psi)\leq C_2;$ 
\item[(3)]  
$\|H_\psi\|_1\leq C_3$.  
\end{enumerate} 
Then $g_\psi$ defined by \eqref{lpop}
 is bounded on $L^p_w:$   
\begin{equation}\label{eq1}
\|g_\psi(f)\|_{p,w} \leq C\|f\|_{p,w} \quad 
\text{for all $p \in (1, \infty)$ and $w \in A_p$, }
\end{equation}  
where the constant $C$ depends only on $p$, $w$, $\epsilon$, $u$ 
and $C_j$, $1\leq j\leq 3$ and does not depend on $\psi$ in the 
other regards. 
If we further assume the non-degeneracy condition$:$  
\begin{equation}\label{parabnondeg} 
\sup_{t>0} |\hat{\psi}(\delta_t^*\xi)|>0 \quad \text{for $\xi\neq 0$,}
\end{equation} 
then we also have the reverse inequality of \eqref{eq1} and hence 
$$ 
\|g_\psi(f)\|_{p,w}\simeq \|f\|_{p,w} 
\quad \text{for all $p \in (1, \infty)$ and $w \in A_p$.}$$
\end{theorem}   
By \cite[Theorem $1.1$]{Sacze}, which generalizes a result of \cite{Sa} to 
the case of non-isotropic dilations, we have the boundedness \eqref{eq1} 
under the conditions (1), (2), (3) of Theorem \ref{thm5} and the quantitative 
property of the constant $C$ specified follows by checking the proof given in 
\cite{Sacze}.  The proof of \cite[Theorem $1.1$]{Sacze} is based on 
estimates for certain oscillatory integrals in  \cite{Sa3}. 

\begin{remark}\label{re2.2} 
If there exist positive numbers $\sigma_1, \sigma_2$ such that 
$$|\psi(x)|\leq C(1+\rho(x)^{-1})^{\gamma -\sigma_1}
(1+\rho(x))^{-\gamma-\sigma_2} \quad \text{for all $x\in \Bbb R^n$}, $$  
then the conditions (1), (2), (3) of Theorem \ref{thm5} are satisfied with 
some $\epsilon$, $u$ and $C_j$, $1\leq j\leq 3$.  To see this the formula 
\eqref{1.9polar} is useful. 
\end{remark}
\par 
To prove the reverse inequality of \eqref{eq1},  
we apply a result on the invertibility on the weighted 
$L^p$ spaces of Fourier multipliers homogeneous  with respect to $\delta_t^*$. 
Let $m\in L^\infty(\Bbb R^n)$,  $w\in A_p$, $1<p<\infty$.  
The Fourier multiplier operator $T_m$ is defined by 
\begin{equation}\label{fmo}
T_m(f)(x)=\int_{\Bbb R^n} m(\xi)\hat{f}(\xi)
e^{2\pi i \langle x,\xi\rangle}\, d\xi.    
\end{equation} 
We say that $m$ is a Fourier multiplier for $L^p_w$ and write
 $m\in M^p_w$ (we also write $M^p(w)$ for $M^p_w$) 
if there exists a constant $C>0$ such that 
\begin{equation}\label{eq2}
\|T_m(f)\|_{p, w}\leq C\|f\|_{p,w} \quad 
\text{for all $f\in \mathscr S$.}
\end{equation} 
We define $\|m\|_{M^p(w)}$ to be the infimum of the constants $C$ 
 satisfying \eqref{eq2}.   
Since $\mathscr S$ is dense in $L^p_w$,  we have a unique extension of 
$T_m$ to a bounded linear operator on $L^p_w$ if $m\in M^p_w$.  
We observe that $M^p(w)=M^{p'}(\widetilde{w}^{-p'/p})$ by duality, where 
 $\widetilde{w}(x)=w(-x)$. 
See \cite{KW} for relevant results.  
\par  
We need the following result generalizing \cite[Theorem 2.5]{Sa4} to 
the case of non-isotropic dilations. 
\begin{theorem}\label{thm6}  
Let $m$ be a bounded function on $\Bbb R^n$ which is continuous on 
$\Bbb R^n \setminus\{0\}$.  Suppose that $m$ is homogeneous of degree $0$ 
with respect to $\delta^*_t$ and that $m(\xi)\neq 0$ for all 
$\xi\in \Bbb R^n \setminus\{0\}$. 
Also, suppose that $m\in M^r_v$ for all $r\in (1,\infty)$ and all $v\in A_r$. 
Let $1<p<\infty,  w\in A_p$ and 
let $F(z)$ be holomorphic in $D=\Bbb C\setminus\{0\}$.  
Then  
$F(m(\xi))\in M^p_w$. 
\end{theorem} 
\par 
For $m\in M^p_w$, $1<p<\infty$, $w\in A_p$, we consider the spectral radius 
operator  
\begin{equation*}
\rho_{p,w}(m)=\lim_{k\to \infty}\|m^k\|_{M^p(w)}^{1/k}.  
\end{equation*}  
To prove Theorem \ref{thm6}, we need the following.  
\begin{proposition}\label{prop1} 
  Suppose that $1<p<\infty$,  $w\in A_p$ and 
$m\in L^\infty(\Bbb R^n)$.  
Let $m$ be homogeneous of degree $0$ with respect to the dilations 
$\delta_t^*$ and continuous on $S^{n-1}$. 
We assume that  $m\in M^r_v$ for all $r\in (1, \infty)$ and all $v\in A_r$.  
Then, for any $\epsilon>0$, there exists $\ell \in M^p_w$ 
which is homogeneous of degree $0$ with respect to $\delta_t^*$ and in 
$C^\infty(\Bbb R^n\setminus\{0\})$ such that $\|m-\ell\|_{\infty}<\epsilon$ 
and $\rho_{p,w}(m-\ell)<\epsilon$.  
\end{proposition} 
To prove Proposition \ref{prop1}, we apply the following lemmas. 
\begin{lemma}\label{lem1}
 Let $\eta\in C^\infty(\Bbb R)$, $\supp \eta\subset [1,2]$, $\eta\geq 0$ and 
$\int_0^\infty |\eta(t)|^2\, dt/t=1$. Define a real function $\psi$ in 
$\mathscr S_0(\Bbb R^n)$ 
by $\hat{\psi}(\xi)=\eta(\rho^*(\xi))$. Then 
$$\|g_\psi(f)\|_{p,w}\simeq \|f\|_{p,w} \quad \text{for all $p\in (1,\infty)$ 
and $w\in A_p$.} $$
\end{lemma} 

\begin{lemma}\label{lem2} 
Suppose that $m\in L^\infty(\Bbb R^n)$, $m\in C^\infty(\Bbb R^n\setminus \{0\})$ and that $m$ is homogeneous of degree $0$ with respect to $\delta_t^*$.  Then 
$m\in M^p_w$ for all $p\in (1,\infty)$ and $w\in A_p$ and 
$$\|m\|_{M^p(w)}\leq C\sup_{1\leq \rho^*(\xi)\leq 2, |a|\leq [\gamma]+1} 
\left|(\partial_\xi)^a m(\xi)\right|  $$  
with a constant $C$ independent of $m$,  
where $(\partial_\xi)^a=
(\partial/\partial \xi_1)^{a_1}\dots (\partial/\partial \xi_n)^{a_n}$ 
with $a=(a_1, \dots ,a_n)$, $a_j\in \Bbb Z$, $a_j\geq 0$, $1\leq j\leq n$.  
\end{lemma} 

\begin{proof}[Proof of Lemma $\ref{lem1}$]  
 By \cite[Theorem 1.1]{Sacze} we see that 
$\|g_\psi(f)\|_{p,w}\leq C\|f\|_{p,w}$ 
for all $p\in (1,\infty)$ and $w\in A_p$. To prove the reverse inequality 
 we note that 
$\|g_\psi(f)\|_2=\|f\|_2$. Thus the polarization implies that for real-valued 
$f, h \in \mathscr S$  
\begin{align*}
4\int_{\Bbb R^n}f(x)h(x)\, dx&=\int_{\Bbb R^n}(f(x)+h(x))^2\, dx
-\int_{\Bbb R^n}(f(x)-h(x))^2\, dx 
\\ 
&= \int_{\Bbb R^n}(g_\psi(f+h)(x))^2\, dx 
-\int_{\Bbb R^n}(g_\psi(f-h)(x))^2\, dx 
\\ 
&= 4\int_{\Bbb R^n}\int_0^\infty f*\psi_t(x)h*\psi_t(x)\, \frac{dt}{t}\, dx. 
\end{align*} 
Therefore,  by the inequalities of Schwarz and H\"older we have 
$$\left|\int_{\Bbb R^n}f(x)h(x)\, dx\right|
\leq \|g_\psi(f)\|_{p,w}\|g_\psi(h)\|_{p',w^{-p'/p}}
\leq C\|g_\psi(f)\|_{p,w}\|h\|_{p',w^{-p'/p}}.  $$ 
Taking the supremum in $h$ with $\|h\|_{p',w^{-p'/p}}\leq 1$, we have 
$\|f\|_{p,w}\leq C\|g_\psi(f)\|_{p,w}$, from which we can derive the desired 
estimates for complex valued functions. 
\end{proof} 

\begin{proof}[Proof of Lemma $\ref{lem2}$]   
Let $\psi$ be as in Lemma \ref{lem1} and define $\psi_m$ by 
$\mathscr F(\psi_m)(\xi)=\hat{\psi}(\xi)m(\xi)$.   Then 
$g_\psi(T_mf)=g_{\psi_m}(f)$. So, by Lemma \ref{lem1} for $w\in A_p$, 
$1<p<\infty$, we have 
\begin{equation}\label{lem2e1}
\|T_mf\|_{p,w}\leq C\|g_\psi(T_mf)\|_{p,w}=C\|g_{\psi_m}(f)\|_{p,w}
\end{equation} 
Since $\psi_m\in \mathscr S_0$,  $g_{\psi_m}$ is bounded on $L^p_w$.   
To specify the operator bounds, we apply the estimates \eqref{eq1}.  
It is 
sufficient to observe the following estimates: 
\begin{multline}\label{lem2e2}  
|\psi_m(x)|=\left|\int_{\Bbb R^n}\hat{\psi}(\xi)m(\xi) 
e^{2\pi i\langle x,\xi \rangle}\,d\xi\right| 
\\ 
\leq C(1+|x|)^{-[\gamma]-1}\sup_{1\leq \rho^*(\xi)\leq 2, |a|\leq [\gamma]+1} 
\left|(\partial_\xi)^a m(\xi)\right|, 
\end{multline}
which follows by integration by parts,  with the constant $C$  
independent of $m$. 
Combining \eqref{lem2e1}, \eqref{lem2e2} and the estimates \eqref{eq1}, 
we have the conclusion. 
\end{proof}

\begin{proof}[Proof of Proposition $\ref{prop1}$] 
As in \cite{H}, \cite{Sa4},  
we take a sequence of functions $\{\varphi_j\}_{j=1}^\infty$ on the orthogonal 
group $O(n)$ with the following properties (1) and (2): 
\begin{enumerate} 
\item[(1)]  functions $\varphi_j$ are infinitely differentiable, 
non-negative and  satisfy 
$\int_{O(n)} \varphi_j(A)\, dA= 1$, where $dA$ is the Haar measure on $O(n)$. 
\item[(2)] for any neighborhood $U$ of the identity of $O(n)$,  there 
exists a positive integer $N$ such that $\supp(\varphi_j)\subset U$ 
for $j\geq N$.    
\end{enumerate} 
For $\xi\in S^{n-1}$, let 
$$\widetilde{m}_j(\xi)= \int_{O(n)} m(A\xi)\varphi_j(A)\, dA.  $$  
Then $\widetilde{m}_j$ is $C^\infty$ on $S^{n-1}$ (see \cite[pp. 123--124]{H}). For 
$\xi \in \Bbb R^n\setminus \{0\}$, let 
$$m_j(\xi)=\widetilde{m}_j\left(\delta^*_{\rho^*(\xi)^{-1}}\xi\right).  $$ 
Then $m_j$ is homogeneous of degree $0$ with respect to $\delta_t^*$, $m_j\in 
C^\infty(\Bbb R^n\setminus \{0\})$ and $m_j=\widetilde{m}_j$ on $S^{n-1}$.   
\par 
We prove   
\begin{equation}\label{prop1e1} 
\rho_{r,v}(m_j)\leq \|m\|_\infty, \quad r\in (1,\infty), v\in A_r. 
\end{equation} 
For this it suffices to show that 
\begin{equation*} 
\|m_j^k\|_{M^r(v)}\leq C_j k^{[\gamma]+1}\|m\|_\infty^k,  
\end{equation*} 
where $C_j$ is independent of $k$. This follows by Lemma \ref{lem2}, since 
$$\sup_{1\leq \rho^*(\xi)\leq 2, |a|\leq [\gamma]+1} 
\left|(\partial_\xi)^a m_j(\xi)^k\right|\leq C_jk^{[\gamma]+1}\|m\|_\infty^k. 
$$ 
To see this,  it is helpful to refer to \cite[pp. 123--124]{H}. 
\par 
Since $m_j \to m$ as $j\to \infty$ uniformly on $S^{n-1}$, we can take 
$\ell=m_j$ for j large enough to get $\|m-\ell\|_\infty<\epsilon$. 
Let $p\in (1,\infty)$, $w\in A_p$. 
Confirming that a result analogous to \cite[Proposition 2.2]{Sa4} holds true 
in the setting of non-isotropic dilations,  
we can find $r>1$, $s>1$ and $\theta\in (0,1)$ such that $w^s\in A_r$ and 
$$ \|(m-m_j)^k\|_{M^p(w)}\leq \|(m-m_j)^k\|_{\infty}^{1-\theta}
\|(m-m_j)^k\|_{M^r(w^s)}^\theta. 
$$
Thus 
$$\rho_{p,w}(m-m_j)\leq \|m-m_j\|_\infty^{1-\theta}\rho_{r,w^s}(m-m_j)^\theta.  $$ 
Since 
$$\rho_{r,w^s}(m-m_j)\leq \rho_{r,w^s}(m)+\rho_{r,w^s}(m_j) $$ 
(see Riesz-Nagy \cite[p. 426]{RN}), it follows that   
\begin{align*}
\rho_{p,w}(m-m_j)&\leq \|m-m_j\|_\infty^{1-\theta}
(\rho_{r,w^s}(m)+\rho_{r,w^s}(m_j))^\theta 
\\ 
&\leq \|m-m_j\|_\infty^{1-\theta}
(\rho_{r,w^s}(m)+\|m\|_\infty)^\theta,  
\end{align*} 
where the last inequality follows from \eqref{prop1e1}. 
Since $\|m-m_j\|_\infty \to 0$ as $j\to \infty$, for a given $\epsilon>0$, 
taking $\ell =m_j$ with $j$ large enough, we have  
$\rho_{p,w}(m-\ell)<\epsilon$ and $\|m-\ell\|_\infty <\epsilon$. 
\end{proof}
\par 
\begin{proof}[Proof of Theorem $\ref{thm6}$]  
The proof is similar to that of \cite[Theorem 2.5]{Sa4}.  
Let   
$$\epsilon_0=\frac{1}{4}\min_{\xi\in S^{n-1}}|m(\xi)|. 
$$ 
Applying Proposition \ref{prop1}, we can find $\ell \in M^p_w$ which is  
homogeneous of degree $0$ with respect to $\delta^*_t$ 
and belongs to $C^\infty(\Bbb R^n\setminus\{0\})$ such that 
$\|m-\ell\|_{\infty}<\epsilon_0$ and $\rho_{p,w}(m-\ell)<\epsilon_0$.  
Let $C: \ell(\xi)+2\epsilon_0e^{i\theta}, 0\leq \theta\leq 
2\pi$, be a circle in $D$. Apply
Cauchy's formula to get   
\begin{equation}\label{pfthm6e1}
F(m(\xi))= \frac{1}{2\pi i}\int_C\frac{F(\zeta)}{\zeta-m(\xi)}\, 
d\zeta
=\frac{\epsilon_0}{\pi}\int_0^{2\pi}
\frac{F(\ell(\xi)+2\epsilon_0e^{i\theta})}{2\epsilon_0e^{i\theta}+
\ell(\xi)-m(\xi)}e^{i\theta}\, d\theta 
\end{equation}  
for $\xi\in \Bbb R^n\setminus\{0\}$.    We expand the integrand in the 
last integral into a power series by using 
\begin{equation}\label{pfthm6e2}
\frac{e^{i\theta}}{2\epsilon_0e^{i\theta}+\ell(\xi)-m(\xi)}
=\frac{1}{2\epsilon_0}\sum_{k=0}^\infty
\left(\frac{m(\xi)-\ell(\xi)}{2\epsilon_0e^{i\theta}}\right)^k,  
\end{equation}   
where the series converges uniformly in $\theta\in [0,2\pi]$ since 
$$\left|\frac{m(\xi)-\ell(\xi)}{2\epsilon_0e^{i\theta}}\right|\leq \frac{1}{2}. $$ 
Substituting \eqref{pfthm6e2} in  \eqref{pfthm6e1}, we have 
\begin{equation}\label{pfthm6e3}
F(m(\xi))=\frac{1}{2\pi}\sum_{k=0}^\infty 
\left(\frac{m(\xi)-\ell(\xi)}{2\epsilon_0}\right)^k N_k(\xi), 
\end{equation} 
 where 
\begin{equation*} 
N_k(\xi)
 =\int_0^{2\pi}F(\ell(\xi)+2\epsilon_0e^{i\theta})e^{-ik\theta}
 \, d\theta 
\end{equation*} 
and the series on the right hand side of \eqref{pfthm6e3} converges uniformly 
in $\xi\in \Bbb R^n\setminus\{0\}$, 
since 
$$   \left|\frac{m(\xi)-\ell(\xi)}{2\epsilon_0}\right|\leq \frac{1}{2}, 
\qquad 
\epsilon_0\leq |\ell(\xi)+2\epsilon_0e^{i\theta}|\leq \|m\|_\infty+ 
3\epsilon_0. 
$$ 
Also, 
$N_k(\xi)$ is homogeneous of degree $0$ with respect to $\delta^*_t$ 
and infinitely differentiable in $\Bbb R^n\setminus \{0\}$ and  
$$\sup_{1\leq \rho^*(\xi)\leq 2, |a|\leq [\gamma]+1} 
\left|(\partial_\xi)^a N_k(\xi)\right|\leq C
$$ 
with $C$ independent of $k$. Therefore, by Lemma \ref{lem2} we have 
$\|N_k\|_{M^p(w)}\leq C$ with a constant $C$ independent of $k$.   
Thus we see that 
$$ 
\sum_{k=0}^\infty 
(2\epsilon_0)^{-k}\|(m-\ell)^k\|_{M^p(w)}\|N_k\|_{M^p(w)} 
\leq C\sum_{k=0}^\infty (2\epsilon_0)^{-k}\|(m-\ell)^k\|_{M^p(w)}
$$
and that the last series converges since  
$\|(m-\ell)^k\|_{M^p(w)}\leq \epsilon_0^k$ if 
$k$ is sufficiently large.  
From this and \eqref{pfthm6e3} we can infer that $F(m)\in M^p_w$.   
This completes the proof.   
\end{proof}

\par 
 By Theorem \ref{thm6} in particular we have the following. 
\begin{corollary}\label{inverse}  Let  $1<p<\infty$ and $w\in A_p$. 
Suppose that $m$ is homogeneous of degree $0$ with respect to $\delta^*_t$ 
and that 
$m\in M^r_v$ for all 
$r\in (1,\infty)$ and all $v \in A_r$.  
We  further assume that $m$ is continuous on $S^{n-1}$ and does not 
vanish there. Then $m^{-1}\in M^p_w$.   
\end{corollary}
\begin{proof} 
Take $F(z)=1/z$ in Theorem \ref{thm6}. 
\end{proof}

Applying Corollary \ref{inverse} in the theory of the Littlewood-Paley 
functions, we can prove the following. 

\begin{theorem}\label{thm7}  
Let $\psi \in L^1(\Bbb R^n)$ satisfy \eqref{cancell}.  
Suppose that  $\|g_\psi(f)\|_{r,v}\leq C_{r,v}\|f\|_{r,v}$, 
$f\in \mathscr S$, for all 
$r\in (1,\infty)$ and all $v\in A_r$  and that 
$m(\xi)=\int_0^\infty|\hat{\psi}(\delta^*_t\xi)|^2\, dt/t$ is continuous and 
strictly positive on $S^{n-1}$. Let $f\in \mathscr S$. 
Then we have
$$\|f\|_{p,w}\leq C_{p,w} \|g_\psi(f)\|_{p,w}  $$  
 for all $p\in (1,\infty)$ and all $w\in A_p$. 
\end{theorem} 
To prove Theorem \ref{thm7}, we also need the following lemma. 
\begin{lemma}\label{lem3}  
Suppose that $\|g_\psi(f)\|_{r,v}\leq C_{r,v}\|f\|_{r,v}$, 
$f\in \mathscr S$,  for all 
$r \in (1, \infty)$ and all $v\in A_r$. Then, if  
$m(\xi)$ is defined as in Theorem $\ref{thm7}$ 
and if $1<p<\infty$,  $w\in A_p$, we have  
$m\in M^p_w$.  
\end{lemma} 

\begin{proof}  
For $\epsilon\in (0,1)$, let 
\begin{equation*} 
\Psi^{(\epsilon)}(x)=
\int_\epsilon^{\epsilon^{-1}}\int_{\Bbb R^n}  
\psi_t(x+y)\bar{\psi}_t(y)\,dy\,\frac{dt}{t},   
\end{equation*} 
where $\bar{\psi}_t$ denotes the complex conjugate.  
We note that 
\begin{equation*} 
\mathscr F(\Psi^{(\epsilon)})(\xi)
=\int_\epsilon^{\epsilon^{-1}}\hat{\psi}(\delta^*_tt\xi)\widehat{\bar{\psi}}
(-\delta^*_t\xi)\,\frac{dt}{t}
=\int_\epsilon^{\epsilon^{-1}}|\hat{\psi}(\delta^*_t\xi)|^2 \,\frac{dt}{t}
=: m^{(\epsilon)}(\xi). 
\end{equation*} 
Therefore $\Psi^{(\epsilon)}*f= T_{m^{(\epsilon)}}f$. We observe that 
$$ \Psi^{(\epsilon)}*f(x)=
\int_\epsilon^{\epsilon^{-1}}
\int_{\Bbb R^n} \psi_t*f(y)\bar{\psi}_t(y-x)\,dy\,\frac{dt}{t};  
$$ 
$$ 
\int_{\Bbb R^n}\Psi^{(\epsilon)}*f(x)h(x)\, dx=
\int_\epsilon^{\epsilon^{-1}}
\int_{\Bbb R^n} \psi_t*f(y)\bar{\psi}_t*h(y)\,dy\,\frac{dt}{t}
$$  
for $f, h \in \mathscr S$. Thus by the inequalities of 
Schwarz and H\"older we have 
\begin{align*} 
\left|\int_{\Bbb R^n}\Psi^{(\epsilon)}*f(x)h(x)\, dx\right| 
&\leq \int_{\Bbb R^n} g_\psi(f)(y)g_\psi(\bar{h})(y)\, dy 
\\ 
&\leq \|g_\psi(f)\|_{p,w}\|g_\psi(\bar{h})\|_{p', w^{-p'/p}} 
\\ 
&\leq C\|g_\psi(f)\|_{p,w}\|h\|_{p', w^{-p'/p}}.  
\end{align*}
Taking the supremum over functions $h$ with $\|h\|_{p', w^{-p'/p}}\leq 1$, we 
have 
\begin{equation*}
\|T_{m^{(\epsilon)}}f\|_{p,w}\leq C\|g_\psi(f)\|_{p,w}.  
\end{equation*}  
Letting $\epsilon\to 0$ and noting $ m^{(\epsilon)}\to m$, we have 
\begin{equation}\label{pflem3}
\|T_{m}f\|_{p,w}\leq C\|g_\psi(f)\|_{p,w}.  
\end{equation} 
Since $\|g_\psi(f)\|_{p,w}\leq C\|f\|_{p,w}$, we see that $m\in M^p_w$.  
\end{proof}  

\begin{proof}[Proof of Theorem $\ref{thm7}$] 
Let $m$ be as in Theorem \ref{thm7}. Then by Lemma \ref{lem3}, $m\in 
M^p_w$ for all $p\in (1, \infty)$ and $w\in A_p$.  So we can apply 
Corollary \ref{inverse} to $m$ to conclude that $m^{-1} \in M^p_w$ 
if $1<p<\infty$, $w\in A_p$ and hence using \eqref{pflem3}, we have 
$$\|f\|_{p,w}=\|T_{m^{-1}}T_m f\|_{p,w}\leq C\|T_m f\|_{p,w}\leq 
C\|g_\psi(f)\|_{p,w}$$ 
for $f\in \mathscr S$, which implies the conclusion. 
\end{proof}

\begin{proof}[Proof of Theorem $\ref{thm5}$]  
It remains to prove the reverse inequality of \eqref{eq1}.  
If $m(\xi)=\int_0^\infty|\hat{\psi}(\delta^*_t\xi)|^2\, dt/t$, then by 
the non-degeneracy 
\eqref{parabnondeg} we have $m(\xi)\neq 0$ for $\xi\neq 0$. Therefore,  by 
Theorem \ref{thm7} we have only to show that $m$ is continuous on $S^{n-1}$. 
In \cite{Sacze}, it has been shown that 
$$\int_{2^k}^{2^{k+1}}|\hat{\psi}(\delta^*_t\xi)|^2\, \frac{dt}{t} \leq 
C\min\left(\left|\delta^*_{2^k}\xi\right|^\epsilon, 
\left|\delta^*_{2^k}\xi\right|^{-\epsilon}\right)  $$  
for $\xi\in S^{n-1}$ and $k\in \Bbb Z$ with some $\epsilon>0$ (see 
 \cite[Lemmas $3.1$ and $3.3$]{Sacze}). 
By analogues for $\delta^*_t$ of (a), (b) for $\delta_t$ 
 in Section \ref{s1}, it follows that 
 $$\int_{2^k}^{2^{k+1}}|\hat{\psi}(\delta^*_t\xi)|^2\, \frac{dt}{t} \leq 
C\min\left(2^{k\epsilon}, 2^{-k\epsilon}\right).   $$  
This implies that 
$$\int_\epsilon^{\epsilon^{-1}}|\hat{\psi}(\delta^*_t\xi)|^2\, \frac{dt}{t} 
\to \int_0^\infty|\hat{\psi}(\delta^*_t\xi)|^2\, \frac{dt}{t} 
\quad \text{as $\epsilon\to 0$}
$$
uniformly in $\xi \in S^{n-1}$.  
We note that 
$\int_\epsilon^{\epsilon^{-1}}|\hat{\psi}(\delta^*_t\xi)|^2\, dt/t $ is 
continuous  on $S^{n-1}$ for each fixed $\epsilon>0$.  Thus the continuity of
 $m$ on $S^{n-1}$ follows by the uniform convergence.   
\end{proof}

\begin{remark} \label{rem2.10} 
Let $\psi^{(j)} \in L^1(\Bbb R^n)$ for $j= 1, 2, \dots, \ell$. 
Suppose that  $\psi^{(j)}$ satisfies \eqref{cancell} and 
$(1)$, $(2)$ and $(3)$ of Theorem $\ref{thm5}$ for every $j$, 
$1\leq j\leq \ell$.  
Let  
\begin{equation*} 
\Psi(x)=\left(\psi^{(1)}(x), \dots ,\psi^{(\ell)}(x)\right),    
\end{equation*}  
\begin{equation*} 
\Psi_t(x)=\left(\psi^{(1)}_t(x), \dots ,\psi^{(\ell)}_t(x)\right), \quad 
\mathscr{F}\left(\Psi_t\right)(\xi)=
\left(\mathscr{F}\left(\psi^{(1)}_t\right)(\xi), \dots ,
\mathscr{F}\left(\psi^{(\ell)}_t\right)(\xi)\right).   
\end{equation*}  
We further assume that 
\begin{equation}\label{ellparabnondeg} 
\sup_{t>0}\left|\mathscr{F}\left(\Psi_t\right)(\xi)\right| 
= \sup_{t>0} \left(\sum_{j=1}^\ell 
\left|\mathscr{F}\left(\psi^{(j)}\right)(\delta_t^*\xi)\right|^2\right)^{1/2} 
>0, \quad \forall \xi\in \Bbb R^n\setminus \{0\}. 
\end{equation} 
Let 
\begin{equation*} 
f*\Psi_t(x)=\left(f*\psi^{(1)}_t(x), \dots , f*\psi^{(\ell)}_t(x)\right)  
\end{equation*}  
and 
\begin{equation*} 
g_\Psi(f)(x)=\left(\int_0^\infty |f*\Psi_t(x)|^2\, \frac{dt}{t}\right)^{1/2}, 
\quad 
|f*\Psi_t(x)|=\left(\sum_{j=1}^\ell  \left|f*\psi^{(j)}_t(x)\right|^2
\right)^{1/2}.  
\end{equation*}  
Then by Theorem $\ref{thm5}$ we have 
$\|g_\Psi(f)\|_{p,w} \leq C\|f\|_{p,w}$.  
We can also prove the reverse inequality by adapting the arguments 
given above when $\ell=1$ for the present situation, 
 applying the non-degeneracy \eqref{ellparabnondeg}. 
Thus we have 
\begin{equation}\label{ellequi}
   \|g_\Psi(f)\|_{p,w}\simeq \|f\|_{p,w}.  
\end{equation}   
\end{remark}  

\begin{aexample}   
We give an example in the case of the Euclidean structures $(\rho(x)=|x|$, 
 $\delta_t(x)=tx)$ for which we can apply Remark $\ref{rem2.10}$ 
 to get the norm equivalence in \eqref{ellequi}.  
Let $P_t(x)$ be the Poisson kernel on the upper half space 
$\Bbb R^n \times (0, \infty)$ defined by 
$$P_t(x)=c_n \frac{t}{(|x|^2+t^2)^{(n+1)/2}}
=\int_{\Bbb R^n}e^{-2\pi t|\xi|}e^{2\pi i\langle x,\xi\rangle}\, d\xi. $$ 
Let $\psi^{(j)}(x) =(\partial/\partial x_j)P_1(x)$, $1\leq j\leq n$. Then  
$$ 
\mathscr{F}\left(\psi^{(j)}\right)(\xi)=2\pi i\xi_j e^{-2\pi |\xi|}.  
$$ 
We can see that all the requirements in Remark $\ref{rem2.10}$ 
for $\psi^{(j)}$, $1\leq j\leq n$, needed in the proof of \eqref{ellequi} 
are fulfilled; in particular,  \eqref{ellparabnondeg} follows from  
$$ \left|\mathscr{F}\left(\Psi_t\right)(\xi)\right| =
2\pi t|\xi| e^{-2\pi t|\xi|}.  $$
 Thus we have \eqref{ellequi} 
for $\Psi=((\partial/\partial x_1)P_1, \dots, (\partial/\partial x_n)P_1)$. 
\end{aexample}  

\section{Proofs of Theorems $\ref{thm3}$ and $\ref{thm4}$}\label{s3}  

We apply the following estimates in proving Theorem \ref{thm3}. 
\begin{lemma}\label{lem0}  
Let $F$ be a function in $C^\infty(\Bbb R^n\setminus\{0\})$ which is 
homogeneous of degree $d$ with respect to $\delta_t$.  
Then, for $\rho(x)\geq 1$ we have  
$$
\left|(\partial_x)^aF(x)\right|\leq C_a \rho(x)^{d-|a|} 
$$ 
for all multi-indices $a$ with a positive constant $C_a$ independent of $x$.  
\end{lemma} 
\begin{proof} 
We write $\delta_t=(\delta_{ij}(t))$, $1\leq i, j\leq n$. 
We have $t^d F(x)=F(\delta_tx)$. 
Differentiating both sides by using the chain rule 
on the right hand side, we have 
$$ 
t^d (\partial_x)^aF(x)= 
\left[\left(\prod_{j=1}^n\left(\sum_{i=1}^n\delta_{ij}(t) 
\partial/\partial x_i\right)^{a_j}\right)F\right](\delta_t x). 
$$  
Substituting $t=\rho(x)^{-1}$ in this equation, we have 
$$ 
|(\partial_x)^aF(x)|\leq C
\left(\sup_{|b|=|a|, \rho(x)=1}|(\partial_x)^bF(x)|\right)
\left(\sup_{1\leq i\leq n, 1\leq j\leq n}\delta_{ij}(\rho(x)^{-1})\right)^{|a|}
\rho(x)^{d}. 
$$ 
This implies what we need, since $|\delta_{ij}(t)|\leq Ct$ for $0<t\leq 1$ by 
(b) of Section \ref{s1}. 
\end{proof} 

\begin{proof}[Proof of Theorem $\ref{thm3}$]   
Let $0<\alpha<\gamma$ and $\mathscr L_\alpha= 
\mathscr F^{-1}(\rho^*(\xi)^{-\alpha})$. Then  $\mathscr L_\alpha$ is 
homogeneous of degree $\alpha-\gamma$ with respect to $\delta_t$ and 
belongs to $C^\infty(\Bbb R^n\setminus\{0\})$ (see \cite[pp. 162--165]{CT2}).  
Let $\psi^{(\alpha)}= \mathscr L_\alpha - \mathscr L_\alpha*\Phi$. Then 
$H_\alpha(f)=g_{\psi^{(\alpha)}}(f)$. 
\par 
We easily see that 
\begin{equation}\label{pfthm2e1} 
 |\psi^{(\alpha)}(x)|\leq C\rho(x)^{\alpha-\gamma} \quad \text{for $\rho(x)
\leq 2$.} 
\end{equation} 
Since 
\begin{equation*} 
\psi^{(\alpha)}(x)
=\int_{\Bbb R^n} (\mathscr L_\alpha(x)-\mathscr L_\alpha(x-y))\Phi(y)\, dy, 
\end{equation*} 
and
\begin{equation*} 
\left|(\partial_x)^a\mathscr L_\alpha(x)\right|\leq C_a\rho(x)^{\alpha-\gamma-
|a|}  \quad \text{for $\rho(x) \geq 2$} 
\end{equation*} 
for all multi-indices $a$ by Lemma \ref{lem0}, 
using Taylor's formula with \eqref{moment}  and noting that $\Phi$ is 
compactly supported, we see that 
\begin{equation}\label{pfthm2e2} 
 |\psi^{(\alpha)}(x)|\leq C\rho(x)^{\alpha-\gamma-[\alpha]-1} \quad \text{for $\rho(x) \geq 2$,} 
\end{equation} 
where $\alpha-\gamma-[\alpha]-1<-\gamma$. 
By \eqref{pfthm2e1}, \eqref{pfthm2e2} and \eqref{1.9polar} it follows that 
$\psi^{(\alpha)}\in L^1$ (see Remark \ref{re2.2}).  
Also, we have 
\begin{equation*} 
\left|\mathscr F(\psi^{(\alpha)})(\xi)\right|= 
\left|\rho^*(\xi)^{-\alpha}(1-\hat{\Phi}(\xi)) \right| 
\leq C\rho^*(\xi)^{-\alpha}|\xi|^{[\alpha]+1} 
\leq C\rho^*(\xi)^{-\alpha+[\alpha]+1} 
\end{equation*} 
for $\rho^*(\xi)\leq 1$ by the analogue for $\rho^*$ of (4) for $\rho$ of 
Section \ref{s1}.  So we have $\mathscr F(\psi^{(\alpha)})(0)=0$, 
i.e., $\int \psi^{(\alpha)}=0$, by which combined with \eqref{pfthm2e1}, 
\eqref{pfthm2e2} and (5) for $\rho$ of Section \ref{s1} 
 we see that the conditions (1), (2), (3) of Theorem \ref{thm5}  are satisfied 
for $\psi^{(\alpha)}$.  Further, it is easy to see that 
$$\sup_{t>0}\left|\mathscr F(\psi^{(\alpha)})(\delta^*_t\xi)\right|>0$$ 
 for $\xi\neq 0$.  Thus all the assumptions of Theorem \ref{thm5} are fulfilled  for $\psi^{(\alpha)}$ and the conclusion of Theorem \ref{thm3} follows by 
applying Theorem \ref{thm5} to $g_{\psi^{(\alpha)}}$.  
\end{proof}

\begin{remark} 
If $\psi^{(\alpha)}$ is as in \eqref{eee1.6}, in the case of 
the Euclidean norm and the ordinary dilation, to prove 
$\|f\|_{p,w}\leq C\|g_{\psi^{(\alpha)}}(f)\|_{p,w}$, 
$0<\alpha<2$, $1<p<\infty$, $w\in A_p$, we can also apply 
the polarization technique as in the proof of Lemma \ref{lem1} 
(see also \cite{Sa2}) 
instead of using Theorem \ref{thm5} with the non-degeneracy 
condition \eqref{parabnondeg}, which is applicable in a more general 
situation of Theorem \ref{thm3}.  
This is the case  because  $\mathscr F(\psi^{(\alpha)})$ is a radial function. 
\end{remark}

\par 
 To prove Theorem \ref{thm4} we prepare the following lemmas. 

\begin{lemma}\label{lem4} Let $1<p<\infty$, $w\in A_p$ and $f\in L^p_w$. 
For a positive integer $m$, let $f_{(m)}=f\chi_{E_m}$, where 
$$E_m=\{x\in \Bbb R^n: |x|\leq m,  |f(x)|\leq m \}. $$
Then $f_{(m)} \to f$ almost everywhere and in $L^p_w$ as $m\to \infty$. 
\end{lemma}

\begin{lemma}\label{lem5} Let $p$, $w$ and $f$ be as in Lemma $\ref{lem4}$. 
Let $\varphi$ be an infinitely differentiable, non-negative function on 
$\Bbb R^n$ such that 
$\varphi(\xi)=1$ for $\rho^*(\xi) \leq 1$, 
$\supp(\varphi)\subset \{\rho^*(\xi) \leq 2\}$ and 
$\varphi(\xi)=\varphi_0(\rho^*(\xi))$ for some $\varphi_0$ 
on $\Bbb R$. 
Define $\zeta^{(\epsilon)}\in \mathscr S_0$ by 
$$\zeta^{(\epsilon)}(\xi)=\varphi(\delta^*_\epsilon\xi)-
\varphi(\delta^*_{\epsilon^{-1}}\xi), \quad 
\epsilon\in (0,1/2). 
$$ 
We note that $\zeta^{(\epsilon)}(\xi)
=\zeta^{(\epsilon/2)}(\xi)\zeta^{(\epsilon)}(\xi)$.  
Let $f^{(\epsilon)}=f*\mathscr F^{-1}(\zeta^{(\epsilon)})$. Then 
$f^{(\epsilon)} \to f$ almost everywhere and in $L^p_w$ as $\epsilon \to 0$. 
\end{lemma}

\begin{proof}[Proof of Lemma $\ref{lem4}$] 
The pointwise convergence is obvious and the norm convergence 
follows from the dominated convergence theorem of Lebesgue since 
$|f_{(m)}|\leq |f|$.  
\end{proof} 

\begin{proof}[Proof of Lemma $\ref{lem5}$] 
If $f\in \mathscr S$, we easily see that $f^{(\epsilon)}\to f$ pointwise 
as $\epsilon \to 0$.  Therefore, for $f\in L^p_w$, we have    
\begin{align*} 
 \left\|\limsup_{\epsilon\to 0}|f^{(\epsilon)}-f|\right\|_{p,w} 
&\leq \left\|\limsup_{\epsilon\to 0}|(f-h)^{(\epsilon)}-(f-h)
|\right\|_{p,w}
\\ 
&\leq C\|M(f-h)\|_{p,w}  
\\ 
&\leq C\|f-h\|_{p,w}  
\end{align*}
for any $h\in \mathscr S$. Since $\mathscr S$ is dense in $L^p_w$, it follows 
that $\limsup_{\epsilon\to 0}|f^{(\epsilon)}(x)-f(x)|=0$ a. e., which implies 
the pointwise convergence.  
The norm convergence follows from the pointwise convergence and the dominated 
convergence theorem of Lebesgue since $|f^{(\epsilon)}|\leq CM(f) \in L^p_w$. 
\end{proof}

\begin{proof}[Proof of Theorem $\ref{thm4}$] 
Define $f_{m,\epsilon}=(f_{(m)})^{(\epsilon)}$ for $f\in L^p_w$. 
Then $f_{m,\epsilon}\in 
\mathscr S_0$. By Theorem \ref{thm3}, we see that 
\begin{equation}\label{pfthm4e1} 
\|G_\alpha(f_{m,\epsilon})\|_{p,w}
=\|H_\alpha(\mathcal{I}_{-\alpha}f_{m,\epsilon})\|_{p,w} 
\simeq \|\mathcal{I}_{-\alpha}^{(\epsilon/2)}f_{m,\epsilon}\|_{p,w},  
\end{equation}
where $\mathcal{I}_{\beta}^{(\epsilon/2)}(f)  
=\mathscr F^{-1}(\zeta^{(\epsilon/2)}(\rho^*)^{-\beta})*f$, $\beta \in 
\Bbb R$,  for $f\in L^p_w$ 
and  we have used the equality $\mathcal{I}_{-\alpha}f_{m,\epsilon}= 
\mathcal{I}_{-\alpha}^{(\epsilon/2)}f_{m,\epsilon}$. 
Using Lemma \ref{lem4}, we see that $f_{m,\epsilon} \to f^{(\epsilon)}$ 
in $L^p_w$, since 
$$ 
\|f_{m,\epsilon}-f^{(\epsilon)}\|_{p,w}\leq C\|M(f_{(m)}-f)\|_{p,w} 
\leq C\|f_{(m)}-f\|_{p,w} 
$$
and also $f_{m,\epsilon} \to f^{(\epsilon)}$ pointwise, since 
\begin{multline*}
\left|f_{m,\epsilon}(x)- f^{(\epsilon)}(x)\right| 
= \left|\int (f_{(m)}(y)-f(y))\mathscr F^{-1}(\zeta^{(\epsilon)})(x-y)\, dy  
\right| 
\\ 
\leq \|f_{(m)}-f\|_{p,w} 
\left(\int \left|\mathscr F^{-1}(\zeta^{(\epsilon)})(x-y)\right|^{p'} 
w(y)^{-p'/p} \, dy \right).    
\end{multline*}  
 Thus $f_{m,\epsilon}-\Phi_t*f_{m,\epsilon}\to 
f^{(\epsilon)}-\Phi_t*f^{(\epsilon)}$ a.e. as $m\to \infty$ and 
by \eqref{pfthm4e1} 
we have, via Fatou's lemma, 
\begin{align*} \label{pfthm4e2}
\|G_\alpha(f^{(\epsilon)})\|_{p,w} &\leq \liminf_{m\to \infty} 
\|G_\alpha(f_{m, \epsilon})\|_{p,w} 
\\ 
&\leq C\liminf_{m\to \infty}\|\mathcal{I}_{-\alpha}^{(\epsilon/2)}
f_{m,\epsilon}\|_{p,w} 
=C\|\mathcal{I}_{-\alpha}^{(\epsilon/2)}f^{(\epsilon)}\|_{p,w},   \notag  
\end{align*} 
where the last equality follows since $\mathcal{I}_{-\alpha}^{(\epsilon/2)}$ 
is bounded on $L^p_w$.    Thus we see that $G_\alpha(f^{(\epsilon)})\in L^p_w$.  In fact, we also have the reverse inequality.  
To see this we first note that  
\begin{equation}\label{pfthm4e6}
\|G_\alpha(f^{(\epsilon)})- G_\alpha(f_{m,\epsilon})\|_{p,w} 
\leq \|G_\alpha(f^{(\epsilon)}-f_{m,\epsilon})\|_{p,w} 
= \|G_\alpha((f-f_{(m)})^{(\epsilon)})\|_{p,w} 
\end{equation} 
Since 
$$  
(f_{(k)}-f_{(m)})^{(\epsilon)}-\Phi_t*(f_{(k)}-f_{(m)})^{(\epsilon)} 
\to (f-f_{(m)})^{(\epsilon)}-\Phi_t*(f-f_{(m)})^{(\epsilon)} \quad 
\text{a.e. as $k\to \infty$}, 
$$
by Fatou's lemma we have  
\begin{equation}\label{pfthm4e7}
\|G_\alpha((f-f_{(m)})^{(\epsilon)})\|_{p,w} \leq \liminf_{k\to \infty}
\|G_\alpha((f_{(k)}-f_{(m)})^{(\epsilon)})\|_{p,w}.  
\end{equation} 
Since $(f_{(k)}-f_{(m)})^{(\epsilon)}\in \mathscr S_0$, by Theorem \ref{thm3} we have 
\begin{multline*} 
\|G_\alpha((f_{(k)}-f_{(m)})^{(\epsilon)})\|_{p,w}\simeq  
\|\mathcal{I}_{-\alpha}((f_{(k)}-f_{(m)})^{(\epsilon)})\|_{p,w} 
\\ 
= \|\mathcal{I}_{-\alpha}^{(\epsilon/2)}((f_{(k)}-f_{(m)})^{(\epsilon)})
\|_{p,w}.
\end{multline*}  
Since $f_{(m)} \to f$ in $L^p_w$, this implies that 
$$ 
\lim_{k, m\to \infty}\|G_\alpha((f_{(k)}-f_{(m)})^{(\epsilon)})\|_{p,w}=0. 
$$   
Thus by \eqref{pfthm4e6} and \eqref{pfthm4e7}, it follows that 
$G_\alpha(f_{m,\epsilon}) \to G_\alpha(f^{(\epsilon)})$ in $L^p_w$ 
as $m\to \infty$. 
Therefore,  letting $m\to \infty$ in \eqref{pfthm4e1}, we have   
\begin{equation}\label{pfthm4e8}
\|G_\alpha(f^{(\epsilon)})\|_{p,w}
\simeq \|\mathcal{I}_{-\alpha}^{(\epsilon/2)}f^{(\epsilon)}\|_{p,w}. 
\end{equation} 
\par 
Suppose that $f\in W^{\alpha, p}_w$ and let $g=\mathcal{I}_{-\alpha}(f)$. 
We show that 
\begin{equation}\label{pfthm4e3}
\mathcal{I}_{-\alpha}^{(\epsilon/2)}f^{(\epsilon)}=g^{(\epsilon)}
\end{equation} 
 as follows.  We have for $h\in \mathscr S_0$  
\begin{align*} 
\int g^{(\epsilon)}\mathcal{I}_{\alpha}(h) \, dx 
&=\lim_{m\to \infty}\int g_{m, \epsilon}\mathcal{I}_{\alpha}(h) \, dx 
=\lim_{m\to \infty}\int \mathcal{I}_{\alpha}^{(\epsilon/2)}(g_{m, \epsilon})h
 \, dx  
\\ 
&=\int \mathcal{I}_{\alpha}^{(\epsilon/2)}(g^{(\epsilon)})h \, dx.    
\end{align*} 
Also, 
\begin{multline*} 
\int g^{(\epsilon)}\mathcal{I}_{\alpha}(h) \, dx=
\lim_{m\to \infty}\int g_{m,\epsilon}\mathcal{I}_{\alpha}(h) \, dx 
\\ 
=\lim_{m\to \infty}\int g_{(m)}\mathcal{I}_{\alpha}(h^{(\epsilon)}) \, dx 
=\int g\mathcal{I}_{\alpha}(h^{(\epsilon)})\, dx.  
\end{multline*} 
By the definition of $g=\mathcal{I}_{-\alpha}(f)$, 
  $\int g\mathcal{I}_{\alpha}(h^{(\epsilon)}) \, dx= 
\int fh^{(\epsilon)}\, dx$.  
 Thus 
\begin{align*} 
\int g^{(\epsilon)}\mathcal{I}_{\alpha}(h) \, dx &=\int fh^{(\epsilon)}\, dx 
=\lim_{m\to \infty}\int f_{(m)}h^{(\epsilon)}\, dx
\\ 
&=\lim_{m\to \infty}\int f_{m,\epsilon}h\, dx
=\int f^{(\epsilon)}h\, dx. 
\end{align*} 
Therefore  
\begin{equation*} 
\int \mathcal{I}_{\alpha}^{(\epsilon/2)}(g^{(\epsilon)})h \, dx 
=\int f^{(\epsilon)}h\, dx  \quad \text{for all $h\in \mathscr S_0$,}
\end{equation*} 
which implies that $\mathcal{I}_{\alpha}^{(\epsilon/2)}(g^{(\epsilon)})=
f^{(\epsilon)}$. Since 
$\mathcal{I}_{\alpha}^{(\epsilon/2)}$ and 
$\mathcal{I}_{-\alpha}^{(\epsilon/2)}$ are bounded on $L^p_w$ and the 
mapping $f\to f^{(\epsilon)}$ is also bounded on $L^p_w$, by Lemma \ref{lem4} 
we see that 
\begin{align*} 
\mathcal{I}_{-\alpha}^{(\epsilon/2)}(f^{(\epsilon)})&=
\mathcal{I}_{-\alpha}^{(\epsilon/2)}
\mathcal{I}_{\alpha}^{(\epsilon/2)}(g^{(\epsilon)}) 
=\lim_{m\to\infty}\mathcal{I}_{-\alpha}^{(\epsilon/2)}
\mathcal{I}_{\alpha}^{(\epsilon/2)}(g_{m,\epsilon}) 
\\ 
&=\lim_{m\to\infty}g_{m,\epsilon} = g^{(\epsilon)}, 
\end{align*} 
which proves \eqref{pfthm4e3}.  
\par 
By \eqref{pfthm4e8} and \eqref{pfthm4e3}, we have 
$$ 
\|G_\alpha(f^{(\epsilon)})\|_{p,w}\leq C\|g^{(\epsilon)}\|_{p,w}
\leq C\|M(g)\|_{p,w}\leq C\|g\|_{p,w}. 
$$ 
Letting $\epsilon \to 0$ and applying  Lemma \ref{lem5} and 
Fatou's lemma, we have 
\begin{equation}\label{pfthm4e4}
 \|G_\alpha(f)\|_{p,w} \leq C\|\mathcal{I}_{-\alpha}(f)\|_{p, w}.  
\end{equation} 
\par 
Conversely, let us assume that $f\in L^p_w$ and $G_\alpha(f)\in L^p_w$.  
By Minkowski's inequality we see that 
\begin{equation}\label{pfthm4e0}
\|G_\alpha(f^{(\epsilon)})\|_{p,w}\leq C\|M(G_\alpha(f))\|_{p,w}
\leq C\|G_\alpha(f)\|_{p,w}. 
\end{equation}  
Applying  \eqref{pfthm4e8} and \eqref{pfthm4e0}, we see that 
$$ \sup_{\epsilon\in (0,1/2)}\|\mathcal{I}_{-\alpha}^{(\epsilon/2)}
f^{(\epsilon)}\|_{p,w}  
\leq C\sup_{\epsilon\in (0,1/2)} \|G_\alpha(f^{(\epsilon)})\|_{p,w} 
\leq C \|G_\alpha(f)\|_{p,w}. $$
Therefore, there exist a sequence $\{\epsilon_k\}$, $0<\epsilon_k<1/2$, 
 and 
a function $g\in L^p_w$ such that $\epsilon_k \to 0$ and 
$\mathcal{I}_{-\alpha}^{(\epsilon_k/2)}f^{(\epsilon_k)} \to g$ weakly in 
$L^p_w$ as $k\to \infty$ and 
\begin{equation}\label{pfthm4e9}   
\|g\|_{p,w} \leq C\|G_\alpha(f)\|_{p,w}.   
\end{equation}   
\par 
We show that $f=\mathcal{I}_{\alpha}g$. By Lemma \ref{lem5}, $f^{(\epsilon_k)} 
\to f$ in $L^p_w$. So, for $h\in \mathscr S_0$ we have 
\begin{align*}  
\int_{\Bbb R^n}fh\, dx&=\lim_{k\to\infty}\int_{\Bbb R^n}f^{(\epsilon_k)}h\, dx 
=\lim_{k\to\infty}\lim_{m\to\infty}\int_{\Bbb R^n}f_{m,\epsilon_k}h\, dx 
\\ 
&=\lim_{k\to\infty}\lim_{m\to\infty}\int_{\Bbb R^n}
\mathcal{I}_{-\alpha}(f_{m,\epsilon_k})\mathcal{I}_{\alpha}(h)\, dx 
\\
&=\lim_{k\to\infty}\lim_{m\to\infty}\int_{\Bbb R^n}
\mathcal{I}_{-\alpha}^{(\epsilon_k/2)}(f_{m,\epsilon_k})\mathcal{I}_{\alpha}(h)\, dx 
\\ 
&=\lim_{k\to\infty}\int_{\Bbb R^n}\mathcal{I}_{-\alpha}^{(\epsilon_k/2)}
(f^{(\epsilon_k)})\mathcal{I}_{\alpha}(h)\, dx 
=\int_{\Bbb R^n}g\mathcal{I}_{\alpha}(h)\, dx.  
\end{align*} 
This implies that $f=\mathcal{I}_{\alpha}g$ by definition. By \eqref{pfthm4e9} 
we have 
$$ \|\mathcal{I}_{-\alpha}f\|_{p,w}=\|g\|_{p,w} 
\leq C\|G_\alpha(f)\|_{p,w}, $$  
which combined with \eqref{pfthm4e4}, completes the proof of 
Theorem \ref{thm4}.  
\end{proof}  

\section{Characterization of the Sobolev spaces $W_w^{\alpha, p}$ 
by square functions defined with repeated uses of 
averaging operation}\label{s4} 

 Let $\Phi\in \mathcal M^1$. 
Define $\Lambda_t^j f(x)$, $j\geq 1$, by $\Lambda_t^j f(x)=f*\Phi^{(j)}_t(x)$, 
where 
\begin{equation*} 
 \Phi^{(1)}(x)=\Phi(x), \qquad 
\Phi^{(j)}(x)=\overbrace{\Phi*\dots *\Phi}^{j}(x), \quad j\geq 2. 
\end{equation*} 
We also write $\Lambda_t f(x)$ for $\Lambda_t^1 f(x)$.  
Let $I$ be the identity operator and $k$ a positive integer. 
We consider 
\begin{align}\label{eeee4.1} 
(I-\Lambda_t)^k f(x)&= f(x)+ \sum_{j=1}^k(-1)^j \binom{k}{j}\Lambda_t^j f(x) 
\\  
&=f(x)-K^{(k)}_t*f(x) =\int_{\Bbb R^n}(f(x)-f(x-y))K^{(k)}_t(y)\, dy,   \notag 
\end{align} 
for appropriate functions $f$, 
where 
\begin{equation}\label{eeee4.2} 
K^{(k)}(x)=-\sum_{j=1}^k(-1)^j \binom{k}{j}\Phi^{(j)}(x), 
\end{equation} 
and we have used the equation 
\begin{equation}\label{eeee4.3} 
\int_{\Bbb R^n} K^{(k)}(x)\, dx=-\sum_{j=1}^k(-1)^j \binom{k}{j}=1. 
\end{equation} 
Define  
 \begin{equation}\label{e4.1} 
 E_\alpha^{(k)} (f)(x)=\left(\int_0^\infty 
\left|(I-\Lambda_t)^k f(x)
\right|^2 \frac{dt}{t^{1+2\alpha}}\right)^{1/2},  \quad \alpha>0. 
 \end{equation} 
If $\Phi=\chi_0=|B(0,1)|^{-1}\chi_{B(0,1)}$ and $k=2$ in \eqref{e4.1}, 
we have  
\begin{equation*}\label{} 
 E_\alpha^{(2)} (f)(x)=\left(\int_0^\infty \left|f(x)-
 2\dashint_{B(x,t)} f(y)\, dy +\dashint_{B(x,t)}(f)_{B(y,t)}\, dy\right|^2 
\frac{dt}{t^{1+2\alpha}}\right)^{1/2}, 
\end{equation*} 
where $(f)_{B(y,t)}=\dashint_{B(y,t)} f$.  
 Also,  let  
 \begin{equation}\label{e4.2} 
 U_\alpha^{(k)} (f)(x)=
\left(\int_0^\infty \left|(I-\Lambda_t)^k\mathcal{I}_\alpha(f)(x)
\right|^2 \frac{dt}{t^{1+2\alpha}}\right)^{1/2}, 
 \end{equation} 
where $0<\alpha<\gamma$, $f\in \mathscr S_0$.  
Using \eqref{eeee4.1}, we can rewrite $E_\alpha^{(k)} (f)$ in \eqref{e4.1} and 
$U_\alpha^{(k)} (f)$ in \eqref{e4.2} as follows: 
\begin{equation}\label{eeeee4.6} 
 E_\alpha^{(k)} (f)(x)=\left(\int_0^\infty 
\left|f(x)-K^{(k)}_t*f(x)
\right|^2 \frac{dt}{t^{1+2\alpha}}\right)^{1/2},  
\end{equation}  
 \begin{equation}\label{eeeee4.7} 
 U_\alpha^{(k)} (f)(x)=
\left(\int_0^\infty \left|\mathcal{I}_\alpha(f)(x)-
K^{(k)}_t*\mathcal{I}_\alpha(f)(x)
\right|^2 \frac{dt}{t^{1+2\alpha}}\right)^{1/2},  
 \end{equation} 
where $K^{(k)}$ is as in \eqref{eeee4.2}. 
 \par 
As applications of Theorems \ref{thm3} and \ref{thm4} 
we have the following theorems.  
\begin{theorem}\label{thm4.1} 
 Let $0<\alpha<\min(2k, \gamma)$, $1<p<\infty$, $w\in A_p$ and let 
$U_\alpha^{(k)}$ be as in \eqref{e4.2}.   
  Then 
$$\|U_\alpha^{(k)} (f)\|_{p,w} \simeq \|f\|_{p,w}, \quad f\in 
\mathscr S_0(\Bbb R^n).     $$  
\end{theorem}  
\par 
\begin{theorem}\label{thm4.2} 
 Let $1<p<\infty$, $w\in A_p$ and $0<\alpha<\min(2k, \gamma)$.  
Let $E_\alpha^{(k)}$ be as in \eqref{e4.1}.  Then 
$f\in W^{\alpha, p}_w$ if and only if $f\in L^p_w$ and 
$E_\alpha^{(k)}(f) \in L^p_w;$  also, we have  
$$\|\mathcal{I}_{-\alpha}(f)\|_{p,w}
\simeq \|E_\alpha^{(k)}(f)\|_{p,w}.  $$
\end{theorem}  

\begin{proof}[Proofs of Theorems $\ref{thm4.1}$ and $\ref{thm4.2}$] 
Using the expressions of $E_\alpha^{(k)} (f)$ and $U_\alpha^{(k)} (f)$ 
in \eqref{eeeee4.6} and \eqref{eeeee4.7}, we can derive 
Theorems $\ref{thm4.1}$ and $\ref{thm4.2}$ from 
Theorems $\ref{thm3}$ and $\ref{thm4}$, respectively, if 
$K^{(k)}\in \mathcal M^{2k-1}$, since then $K^{(k)}\in \mathcal M^{\alpha}$ 
for $\alpha\in (0, \min(2k, \gamma))$.  
\par 
To show that $K^{(k)}\in \mathcal M^{2k-1}$, first we easily see that 
$K^{(k)}$ is bounded and compactly supported. 
Since we have already noted \eqref{eeee4.3}, it remains to show that 
\begin{equation}\label{eeee4.6}
\int_{\Bbb R^n} y^a K^{(k)}(y)\, dy=0 \quad \text{if $1\leq |a|< 2k$.} 
\end{equation} 
 This can be shown 
as follows.   
Since $\Phi\in \mathcal M^1$, 
 we have  $\int y^a \Phi(y)\, dy=0$ for $|a|=1$, which 
implies that $\partial_\xi^a\hat{\Phi}(0)=0$ for $|a|=1$.  
Therefore we have, near $\xi=0$,  
\begin{align}\label{eeee4.7}  
1-\mathscr{F}(K^{(k)})(\xi)=1+\sum_{j=1}^k(-1)^j \binom{k}{j}\hat{\Phi}(\xi)^j
=\left(1-\hat{\Phi}(\xi)\right)^k =O(|\xi|^{2k}).  
\end{align}
Also, by Taylor's formula we see that 
\begin{equation}\label{eeee4.8} 
\mathscr{F}(K^{(k)})(\xi)=
1+\sum_{1\leq |a|<2k} C_a\xi^a \partial_\xi^a\mathscr{F}(K^{(k)})(0)
+ O(|\xi|^{2k}).  
\end{equation} 
From \eqref{eeee4.7} and \eqref{eeee4.8} it follows that 
\begin{equation*} 
\sum_{1\leq |a|<2k} C_a\xi^a \partial_\xi^a\mathscr{F}(K^{(k)})(0)
= O(|\xi|^{2k}).  
\end{equation*}
This implies that
 $\partial_\xi^a\mathscr{F}(K^{(k)})(0)=0$ for $1\leq |a|<2k$, 
and hence we have \eqref{eeee4.6}.  
\end{proof}

\begin{remark}\label{sec4re}  
In the definitions 
of $E^{(k)}_\alpha$ and $U^{(k)}_\alpha$ in \eqref{e4.1} and \eqref{e4.2}, 
if we assume only that $\Phi$ belongs to $\mathcal M^0$, 
then we have analogues of Theorems \ref{thm4.1} and \ref{thm4.2} for 
the range  $(0, \min(k, \gamma))$ of $\alpha$.  

\end{remark} 

\section{The Sobolev spaces $W_w^{\alpha, p}$ and distributional derivatives}
\label{s5} 

In $\Bbb R^2$, we consider $P=\mathop{\mathrm{diag}}(1,2)$, $\delta_t=
 \mathop{\mathrm{diag}}(t,t^2)$.  Then, $\gamma=3$ and  
 $$\rho(x_1,x_2)=\frac{1}{\sqrt{2}}\sqrt{x_1^2+ \sqrt{x_1^4+4x_2^2}},   $$ 
 $\rho^*=\rho$, $\delta_t^*=\delta_t$.  
Under this setting, let $W_w^{\alpha, p}$ be the weighted Sobolev space on 
$\Bbb R^2$ defined in Section \ref{s1} with $0<\alpha<3$, $1<p<\infty$, 
$w\in A_p$.  Then $W_w^{2, p}$ can be characterized by using distributional 
derivatives as follows.

\begin{theorem}\label{thm5.1}  
 Let $f\in L^p_w$ with $1<p<\infty$, $w \in A_p$. 
Let $(\partial/\partial x_1)^2f$, $\partial/\partial x_2 f$ be the 
distributional derivatives in $\mathscr S'$ $($the space of tempered 
distributions$)$. Then, 
$f\in W_w^{2, p}$ if and only if $(\partial/\partial x_1)^2f \in L^p_w$ and  
 $\partial/\partial x_2 f \in L^p_w$; further 
$$ \|\mathcal{I}_{-\alpha}(f)\|_{2,w} \simeq 
\|(\partial/\partial x_1)^2 f\|_{p,w} + \|\partial/\partial x_2 f\|_{p,w}. 
$$ 
\end{theorem}

\begin{proof} 
Suppose that $f\in W_w^{2, p}$. Let $g=\mathcal{I}_{-2}(f) \in 
L^p_w$.  Then we have 
\begin{equation}\label{e5.1} 
\int f h \, dx= \int g \mathcal{I}_{2}(h) \, dx \quad 
\text{for all $h \in \mathscr S_0$.}   
\end{equation}  
Let $k(\xi)=-4\pi^2\xi_1^2$. 
Let $g_{m,\epsilon}=g_{(m)}*\mathscr F^{-1}(\zeta^{(\epsilon)})$ be as in Section \ref{s3}.  
Then by \eqref{e5.1}  we see that for $h \in \mathscr S_0$ 
\begin{align}\label{e5.2} 
&\int f (\partial/\partial x_1)^2h\, dx
=  \int g \mathcal{I}_{2}\left((\partial/\partial x_1)^2h\right) \, dx 
=\int g \mathcal{I}_{2}\left(T_k h\right) \, dx 
\\   \notag 
&=\lim_{\epsilon\to 0}\lim_{m\to \infty} \int g_{m,\epsilon} 
\mathcal{I}_{2}\left(T_k h\right) \, dx 
=\lim_{\epsilon\to 0}\lim_{m\to \infty} \int T_{k(\rho^*)^{-2}}
\left(g_{m,\epsilon}\right) h  \, dx. 
\end{align} 
Since  $k(\rho^*)^{-2}$ is homogeneous of degree $0$ with respect to 
$\delta_t^*$ and infinitely differentiable in $\Bbb R^2\setminus \{0\}$, 
by Lemma \ref{lem2} the multiplier operator $T_{k(\rho^*)^{-2}}$ is
 bounded on $L^p_w$. Thus 
$T_{k(\rho^*)^{-2}}\left(g_{m,\epsilon}\right)\to T_{k(\rho^*)^{-2}}
\left(g\right)$ in $L^p_w$ as $m\to \infty$, $\epsilon \to 0$ since 
$g_{m,\epsilon} \to g$ in $L^p_w$ as $m\to \infty$, $\epsilon \to 0$.  
Therefore, by \eqref{e5.2} we have  
\begin{equation} \label{e5.3} 
\int f (\partial/\partial x_1)^2h\, dx= \int 
T_{k(\rho^*)^{-2}}\left(g\right) h \, dx \quad h\in \mathscr S_0,  
\end{equation}  
which implies that 
\begin{equation}\label{e5.4} 
\int f (\partial/\partial x_1)^2\psi\, dx= \int 
T_{k(\rho^*)^{-2}}\left(g\right) \psi \, dx \quad \text{for all 
$\psi \in \mathscr S$.}  
\end{equation}  
It follows that 
\begin{equation}\label{e5.5} 
(\partial/\partial x_1)^2f= T_{k(\rho^*)^{-2}}\left(g\right) \quad 
\text{in $\mathscr S'$.}     
\end{equation}  
To see \eqref{e5.4}, substitute $\psi-
\mathscr F^{-1}(\varphi(\delta_\epsilon^{-1}\xi)\hat{\psi}(\xi))$ for $h$ in 
\eqref{e5.3}, where $\varphi$ is as in Lemma \ref{lem5}, and let 
$\epsilon \to 0$. 
 
\par 
Let $\ell(\xi)=2\pi i\xi_2$. Then, 
arguing similarly as above 
and noting that $\ell(\rho^*)^{-2}$ is homogeneous of degree 
$0$ with respect to 
$\delta_t^*$ and infinitely differentiable in $\Bbb R^2\setminus \{0\}$, 
we see that $T_{\ell(\rho^*)^{-2}}\left(g\right)\in L^p_w$ and
\begin{equation*}
- \int f \partial/\partial x_2\psi\, dx= \int 
T_{\ell(\rho^*)^{-2}}\left(g\right) \psi \, dx \quad \text{for all 
$\psi \in \mathscr S$,}  
\end{equation*}  
which implies that 
\begin{equation}\label{e5.6} 
\partial/\partial x_2 f= T_{\ell(\rho^*)^{-2}}\left(g\right) \quad 
\text{in $\mathscr S'$.}     
\end{equation}  
Combining \eqref{e5.5} and \eqref{e5.6}, we have 
\begin{equation} \label{e5.7} 
\|(\partial/\partial x_1)^2 f\|_{p,w} + \|\partial/\partial x_2 f\|_{p,w}
\leq C\|g\|_{p,w}=C\|\mathcal{I}_{-2}(f)\|_{p,w}.  
\end{equation} 
\par 
Conversely, suppose that $(\partial/\partial x_1)^2 f=:\Theta \in L^p_w$ and  
$\partial/\partial x_2 f=: \Xi \in L^p_w$. Then, for $h\in \mathscr S_0$ 
we have 
\begin{equation*} 
\int f(\partial/\partial x_1)^2h \, dx =\int \Theta h \, dx, 
\quad 
-\int f\partial/\partial x_2 h \, dx =\int \Xi h \, dx,  
\end{equation*} 
and hence 
\begin{equation} \label{e5.8} 
\int f (T_k h -T_\ell h) \, dx =
\int f\left((\partial/\partial x_1)^2 h -\partial/\partial x_2 h \right)\, dx 
=\int (\Theta +\Xi)h \, dx,    
\end{equation} 
where $k(\xi)$ and $\ell(\xi)$ are as above. 
Let 
$$N(\xi)=\frac{k(\xi)-\ell(\xi)}{\rho^*(\xi)^2}
=\frac{-4\pi^2\xi_1^2 - 2\pi i\xi_2}{\rho^*(\xi)^2}.  
$$ 
Then, substituting $\mathcal{I}_{2}(h)$ for $h$ in \eqref{e5.8}, we have 
\begin{equation} \label{e5.9} 
\int f T_N h \, dx =\int (\Theta +\Xi)\mathcal{I}_{2}(h) \, dx.   
\end{equation} 
We note that the functions $N$ and $\widetilde{N}^{-1}$ are 
homogeneous of degree $0$ with respect to $\delta_t^*$ and 
infinitely differentiable in $\Bbb R^2\setminus \{0\}$, 
where $\widetilde{N}(\xi)=N(-\xi)$. 
 So,  $T_{\widetilde{N}^{-1}}$ is bounded on $L^p_w$ by 
Lemma \ref{lem2}.   Substituting $T_{N^{-1}}h$ for $h$ in \eqref{e5.9}, 
we have 
\begin{equation} \label{e5.10} 
\int f h \, dx =\int (\Theta +\Xi)T_{N^{-1}}\left(\mathcal{I}_{2}(h)\right)  
\, dx 
=\int T_{\widetilde{N}^{-1}}(\Theta +\Xi)\mathcal{I}_{2}(h) 
\, dx,      
\end{equation} 
where the last equality follows as \eqref{e5.3}, since 
$T_{\widetilde{N}^{-1}}$ is bounded on $L^p_w$.  
  By \eqref{e5.10} we see that 
$f\in W^{2,p}_w$ and 
$$\mathcal{I}_{-2}(f)=T_{\widetilde{N}^{-1}}(\Theta +\Xi) $$ 
and 
$$ 
\|\mathcal{I}_{-2}(f)\|_{p,w}\leq C\|\Theta\|_{p,w} +C\|\Xi\|_{p,w} 
=C\|(\partial/\partial x_1)^2 f\|_{p,w}+ C\|\partial/\partial x_2 f\|_{p,w},  
$$  
which combined with \eqref{e5.7} completes the proof of the theorem. 
\end{proof}

We conclude this note with two remarks. 

\begin{remark} \label{s3r1} 
To characterize the Sobolev spaces $W^{\alpha,p}$ (unweighted spaces) 
we can also apply the square 
functions of the Lusin area integral type instead of the Littlewood-Paley 
function type (see \cite{SWYY}). In \cite{Sajf}, certain Sobolev spaces 
($H^1$ Sobolev spaces) were characterized by using certain square functions of 
the Lusin area integral type. 
 The characterization of those Sobolev spaces 
by square functions of the Littlewood-Paley function type analogous to Theorem 
$\ref{thm4}$ is yet to be proved.   
\end{remark}

\begin{remark} \label{s3r2}  
Let us consider another square function of Marcinkiewicz type: 
\begin{equation*}
D_\alpha(f)(x)=\left(\int_{\Bbb R^n}|I_\alpha(f)(x+y)- I_\alpha(f)(x)|^2 
|y|^{-n-2\alpha}\, dy\right)^{1/2},     
\end{equation*} 
where $I_\alpha$ is as in \eqref{1.04}. 
Let $0<\alpha<1$ and $p_0=2n/(n+2\alpha)>1$.  Then it is known that 
the operator $D_\alpha$ is bounded on $L^p(\Bbb R^n)$ if $p_0<p <\infty$ 
(\cite{St2}) and that $D_\alpha$ is of weak type $(p_0,p_0)$ (\cite{F}). 
In \cite{Sascan} analogues of these results were established in the case of 
dilations $\delta_t=t^P$ when $P$ is diagonal.  
\end{remark}

\end{document}